\documentclass[10pt,a4,leqno]{amsart} 
\usepackage{amsthm}
\usepackage{amsmath}
\usepackage{amssymb}
\usepackage{amsfonts}
\usepackage{mathrsfs}
\usepackage{graphicx}
\usepackage{bbm}
\usepackage[
colorlinks=true, citecolor=blue, linkcolor=blue, urlcolor=red]{hyperref}
\usepackage{color}
\usepackage[all]{xy}
\usepackage{comment}
\usepackage{here}
\allowdisplaybreaks

\makeatletter

\@addtoreset{equation}{section}
\makeatother

\makeatletter
\@namedef{subjclassname@2020}{%
  \textup{2020} Mathematics Subject Classification}
\makeatother

\newtheoremstyle{my theoremstyle}
{1.0em}                    
    {1.0em}                    
    {\itshape}                   
    {}                           
    {\scshape}                   
    {.}                          
    {.5em}                       
    {}  

\newtheoremstyle{dfn}
{1.0em}                    
    {1.0em}                    
    {}                   
    {}                           
    {\scshape}                   
    {.}                          
    {.5em}                       
    {}  
    
\theoremstyle{my theoremstyle}
   \newtheorem{thm}{Theorem}[section]
   \newtheorem{lem}[thm]{Lemma}
   \newtheorem{prop}[thm]{Proposition}
   
\theoremstyle{dfn}

\theoremstyle{remark}   
   
   \newtheorem{rmk}[thm]{{\scshape Remark}}

\newcommand{\Z}{\mathbb{Z}}
\newcommand{\Q}{\mathbb{Q}}
\newcommand{\R}{\mathbb{R}}
\newcommand{\C}{\mathbb{C}}

\newcommand{\tw}{\widetilde{\omega}}
\newcommand{\dv}{\operatorname{div}}
\newcommand{\fa}{{f_{\alpha}}}
\newcommand{\fb}{{f_{\beta}}}
\newcommand{\reg}{r_{\mathscr{D}}}
\newcommand{\K}{\kappa}
\newcommand{\ok}{\kappa^{1, 1}}
\newcommand{\s}{\sharp}
\numberwithin{equation}{section}

\date{\today}

\begin{document}
\title[$L$-function of CM elliptic curves and hypergeometric functions]{$L$-function of CM elliptic curves and generalized hypergeometric functions}
\author{Yusuke Nemoto}
\date{\today}
\address{Graduate School of Science and Engineering, Chiba University, 
Yayoicho 1-33, Inage, Chiba, 263-8522 Japan.}
\email{y-nemoto@waseda.jp}

\maketitle

\begin{abstract}
In this paper, we express special values of the $L$-functions of certain CM elliptic curves which are related to Fermat curves in terms of the special values of generalized hypergeometric functions by comparing Bloch's element with Ross's element in the motivic cohomology group.    
\end{abstract}

\section{Introduction}
Let $E_N$ be an elliptic curve over $\Q$ with conductor $N$ and 
$L(E_N, s)$ be the $L$-function of (the first cohomology of) $E_N$. 
By results of Deuring, when $E_N$ has complex multiplication by the integer ring $\mathscr{O}_K$ of an imaginary quadratic field $K$, 
$L(E_N, s)$ is the $L$-function of a Hecke character of $K$.
Therefore, $L(E_N, s)$ is analytically continued to the whole complex plane and satisfies a functional equation
with respect to $s$ and $2-s$. 
In particular, it has a simple zero at $s=0$ and
$$L^*(E_N, 0):=\lim_{s \to 0}s^{-1}L(E_N, s)=\pm \frac{N}{(2\pi)^2}L(E_N, 2).$$

By previous works of Otsubo \cite{Otsubo1}, Rogers-Zudilin \cite{RZ} and Ito \cite{Ito}, 
the special value of the $L$-function $L^*(E_N, 0)$, or equivalently $L(E_N, 2)$ for some $N$ is expressed in terms of special values of the generalized hypergeometric function ${_{3}F_{2}}(x)$ at $x=1$, which is defined by
$${_{3}F_{2}}\left( 
\begin{matrix}
a_1, a_2, a_3 \\
b_1, b_2
\end{matrix}
; x
\right)
=\sum_{n=0}^{\infty} \dfrac{(a_1)_n (a_2)_n (a_3)_n}{(b_1)_n  (b_2)_n}
 \dfrac{x^n}{n !}. 
$$
Here $(\alpha)_n=\Gamma(\alpha+n)/\Gamma(\alpha)$ denotes the Pochhammer symbol. 
In this paper, we treat the cases when $N=36$ and $64$, i.e.
\begin{align*}
&E_{36} \colon y^2=x^3+1, \\
&E_{64} \colon y^2=x^3-4x.
\end{align*}
Note that $E_{36}$ and $E_{64}$ have complex multiplication by $\Z[(-1+\sqrt{-3})/2]$ and $\Z[i]$ respectively.
For $\alpha, \beta, \alpha+\beta \not \in \Z_{\leq 0}$, we put
$$\widetilde{F}(\alpha, \beta)=\left(\frac{\Gamma(\alpha)\Gamma(\beta)}{\Gamma(\alpha+\beta)}\right)^2
{_{3}F_{2}}\left( 
\begin{matrix}
\alpha, \beta, \alpha+\beta-1 \\
\alpha+\beta, \alpha+\beta
\end{matrix}
; 1
\right).$$
The main theorem of this article is as follows:

\begin{thm}[Theorems \ref{main1} and \ref{main2}] \label{main}
\begin{align*}
L^*(E_{36}, 0)&=\frac{1}{2\sqrt{3}\pi}\left(\widetilde{F}\left( \frac12, \frac13 \right) - \widetilde{F}\left( \frac12, \frac23 \right) \right),\\
L^*(E_{64}, 0)&=\frac{1}{8\pi}\left(\widetilde{F}\left( \frac14, \frac14 \right) - \widetilde{F}\left( \frac34, \frac34 \right) \right).
\end{align*}
\end{thm}

There are two different ways of proving such a relation between $L(E^*_N, 0)$ and ${}_3F_2(1)$; the analytic method (Rogers-Zudilin method)\cite{RZ} 
reducing $L(E^*_N, 0)$ to an integral of elementary functions by some transformations  
and the geometric method \cite{Otsubo1} using the real regulators. 
The following table is a summary of the known results: 
\begin{table}[H]
\begin{tabular}{|c||c|c|} \hline
  $N$ & Analytic method & Geometric method \\ \hline
   27 & Rogers-Zudilin \cite[Theorem 1]{RZ}  & Otsubo \cite[Theorem 5.2]{Otsubo1} \\
   32 & Ito \cite[Theorem 2.5]{Ito} & Otsubo \cite[Theorem 5.5]{Otsubo1} \\
   36 & N/A & Theorem \ref{main1}  \\
   64 & Ito \cite[Theorem 2.7]{Ito}  & Theorem \ref{main2} \\ \hline 
 \end{tabular}
\end{table}
\noindent The relation between $L^*(E_{64}, 0)$ and ${}_3F_2(1)$ is already known by the analytic method, however 
 the relation between $L^*(E_{36}, 0)$ and ${}_3F_2(1)$ is only predicted numerically by Otsubo \cite[Section 4.1]{Otsubo3}.

We briefly explain the geometric method.  
Let $X_{N'}$ be  the Fermat curve of degree $N'$ over $\Q$. 
Let $E_N$ be an elliptic curve of conductor $N$ in the table above, which is a Fermat quotient, i.e. 
there exists a surjective morphism 
$$p\colon X_{N'} \to E_{N} \quad \text{with} \quad (N', N)=(3, 27), (4, 32), (6, 36), (4, 64). $$
Let $e_{N'}$ be Ross's element \cite{Ross} in the integral part of the motivic cohomology $H_{\mathscr{M}}^2(X_{N'}, \Q(2))_{\Z}$. 
By projecting this element, we have
$$p_*(e_{N'}) \in H^2_{\mathscr{M}}(E_{N}, \Q(2))_{\Z}$$
in the integral part of the motivic cohomology of $E_N$.  
As $E_N$ has complex multiplication by the integer ring $\mathscr{O}_K$ of an imaginary quadratic field $K$,  
we have Bloch's element $e_{E_{N}}$ in $H_{\mathscr{M}}^2(E_{N}, \Q(2))_{\Z}$, 
whose regulator image is expressed in terms of the special value $L^*(E_N, 0)$.
The Bloch-Beilinson conjecture \cite{Beilinson, Bloch1} implies that $H^2_{\mathscr{M}}(E_{N}, \Q(2))_{\Z} \simeq \Q$, 
hence if we admit this conjecture, 
there is a constant $c \in \Q^*$ such that 
$$ e_{E_{N}}=cp_*(e_{N'})$$
since $p_*(e_{N'})$ is a non-zero element \cite{Otsubo}. 
Otsubo \cite{Otsubo1} compares the regulator images of Bloch's element for $N=27$ (resp. $N=32$) with Ross's element for $N'=3$ (resp. $N'=4$), and proves that
$$\reg(e_{E_{27}})=\reg(p_*(e_3)) \quad ({\rm resp.} \ \reg(e_{E_{32}})=\reg(p_*(e_4))),$$
where $\reg$ is the real regulator map. 
In this paper, we compare Bloch's element for $N=36$ (resp. $N=64$) with Ross's element for $N'=6$ (resp. $N'=4$) using Brunault's method, without taking the regulators 
and 
determine the constant $c$ explicitly (Propositions \ref{keyprop1} and \ref{keyprop2}), i.e. 
$$e_{E_{36}}=2p_*(e_6) \quad ({\rm resp.} \ e_{E_{64}}=2p_*(e_4)).$$
It is known that the regulator of Ross's element is expressed in terms of the special value of the generalized hypergeometric function ${}_3F_2(1)$ \cite{Otsubo}, 
hence we obtain the main theorem. 

This paper is constructed as follows. 
In Section 2, we recall the Fermat motive and the relation between its regulator and special values of the generalized hypergeometric function according to \cite{Otsubo}. 
In Section 3, we recall the relation between the regulator of an elliptic curve with complex multiplication and the special value of the $L$-function according to \cite{Bloch, Deninger, Otsubo1}.  
In Section 4, we compare Bloch's element with Ross's element and prove the main result.

\section{Regulator of Fermat motives}
First, we briefly recall the Beilinson conjecture for $K_2$ of curves (see \cite{DJZ} for a nice survey).  
Let $X$ be a geometrically connected smooth projective curve over $\Q$ and $\reg$ be the regulator map \cite{Beilinson}
$$\reg \colon H^2_{\mathscr{M}}(X, \Q(2))_{\Z}  \to H^2_{\mathscr{D}}(X_{\R}, \R(2)).$$
The source of $\reg$ is the integral part of the motivic cohomology, and we have  (cf. \cite{Nekovar})
$$H^2_{\mathscr{M}}(X, \Q(2))\simeq \operatorname{Ker}\left(K^M_2(\Q(X)) \otimes \Q \xrightarrow{T \otimes \Q} \bigoplus_{x \in X^{(1)}}\kappa(x)^* \otimes \Q \right),$$
where $X^{(1)}$ is the set of closed points of $X$, $\kappa(x)$ is the residue field at $x$ and $T=(T_x)$ is the tame symbol 
defined  by
$$T_x(\{f, g\})=(-1)^{\operatorname{ord}_x(f)\operatorname{ord}_x(g)}\left(\frac{f^{\operatorname{ord}_x(g)}}{g^{\operatorname{ord}_x(f)}}\right)(x).$$
The integral part consists of those elements that can be extended to a regular model of $X$ that is proper and flat over $\Z$, which is known to exist for curves, and the integral part is independent of the model (cf. \cite[Section 3]{DJZ}). 
The target of $\reg$ is the real Deligne cohomology, and we have an isomorphism 
$$H^2_{\mathscr{D}}(X_{\R}, \R(2)) \simeq H^1(X(\C), \R(1))^+,$$
where $\R(1)=(2\pi i) \R$ and $+$ is the $(+1)$-eigenspace with respect to the simultaneous action of the complex conjugation $F_{\infty}$ (infinite Frobenius) acting on $X(\C)$ and the complex conjugation $c_{\infty}$ on the coefficients. 
The Beilinson conjecture \cite{Beilinson} asserts firstly that $\reg \otimes \R$ is an isomorphism, and hence
$$\dim_{\Q} H^2_{\mathscr{M}}(X, \Q(2))_{\Z}=g$$
where $g$ is the genus of $X$. 
Under the first assertion, the conjecture asserts secondly that the first non-vanishing Taylor coefficient of the $L$-function $L(h^1(X), s)$ at $s=0$ (although the analytic continuation in general is highly conjectural) is expressed in terms of the determinant of the regulator matrix, which is well-defined modulo $\Q^*$.

Let $X_N$ be the Fermat curve of degree $N$ over $\Q$  defined by 
$$x_0^N+y_0^N=z_0^N.$$
Its affine equation is written as 
$$x^N+y^N=1 \quad (x=x_0/z_0,\ y=y_0/z_0). $$
Put $K=\Q(\mu_N)$ and $X_{N, K}=X_N \times_{\Q} K$. 
Fix an embedding $K \hookrightarrow \C$ and put $\zeta=\exp(2 \pi i/N)$. 
Let $G_N:=(\Z/N\Z)^{\oplus 2}$ and write an element $(r, s) \in G_N$ also as $g^{r, s}$. 
Let $G_N$ act on $X_{N, K}$ by
 $$g^{r, s} (x, y) = (\zeta^r x, \zeta^s y). $$
 For $(a, b) \in G_N$, let $p_N^{a, b}$ be the projector corresponding to the character $(r, s) \mapsto \zeta^{ar+bs}$. Then the pair $(X_{N, K}, p_N^{a, b})$ defines a motive over $K$ with coefficients in $K$. 
Put $H_N=(\Z/N\Z)^*$ with the natural action on $G_N$ and $I_N=\{(a, b) \in G_N \mid a, b, a+b \neq 0\}$. 
Let $[a, b]$ be the $H_N$-orbit of $(a, b) \in G_N$. 
We define the projector (\cite[Definition 2.3]{Otsubo}) as
$$p_N^{[a, b]}=\sum_{(c, d) \in [a, b]}p_N^{c, d} $$
which is an element of $\Q[G_N]^{H_N}$ by \cite[Lemma 2.6]{Otsubo}, hence defines a motive $X_N^{[a, b]}=(X_N, p_N^{[a, b]})$ over $\Q$ with $\Q$-coefficients. 
Then we have decompositions
$$h^1(X_{N, K}) =\bigoplus_{(a, b) \in I_N}X_N^{a, b},  \quad   h^1(X_N) =\bigoplus_{[a, b] \in H_N \backslash I_N}X_N^{[a, b]} $$
in the category of Chow motives (\cite[Proposition 2.9 (iii) and Proposition 2.11 (iii)]{Otsubo}).
Let $\omega^{a, b}$ be a differential $1$-form defined by
$$\omega^{a, b}= x^{\langle a \rangle} y^{\langle b \rangle -N} \frac{dx}{x},$$
where $\langle a \rangle \in \{1, 2, \ldots, N-1\}$ denotes the representative of $a$. 
We note that $\{\omega^{a, b}  \mid (a, b) \in I_N\}$ is a basis of $H^1(X_N(\C), \C)$ (\cite[Lemma 2.1 (ii)]{Otsubo4}) and 
$\omega^{a, b}$ is holomorphic if and only if $\langle a \rangle +\langle b \rangle <N$. 
On the other hand, there is a unique element $\kappa \in H_1(X_N(\C), \Z)$ such that 
$H_1(X_N(\C), \Z)$ is a cyclic $\Z[G_N]$-module generated by $\kappa$ and 
$$\int_{\kappa}\omega^{a, b}=\dfrac{(1-\zeta^a)(1-\zeta^b)}{N} B\left(\frac{\langle a \rangle}{N}, \frac{\langle b \rangle}{N}\right),$$
where $B(s, t)$ is the Beta function \cite[Theorem 1]{R}.   
Put 
$$\tw^{a, b}=\left(\frac1N B\left(\frac{\langle a \rangle}{N}, \frac{\langle b \rangle}{N}\right)\right)^{-1} \omega^{a, b}.$$
One sees easily that 
$$c_{\infty}\tw^{a, b}=F_{\infty}\tw^{a, b}=\tw^{-a, -b}, $$
hence $H^1(X(\C), \R(1))^+$ is generated by $\tw^{a, b}-\tw^{-a, -b}$. 
Let 
$$e_N=\{1-x, 1-y\} \in H^2_{\mathscr{M}}(X_N, \Q(2))_{\Z}$$
be the Ross element \cite{Ross} and put 
$$e^{[a, b]}_N=p^{[a, b]}_N ( e_N) \in H^2_{\mathscr{M}}(X^{[a, b]}_N, \Q(2))_{\Z}$$
as in \cite[Definition 4.5]{Otsubo}.

\begin{thm}[{\cite[Theorem 4.14]{Otsubo}}] \label{Otsubo}
If $(a, b) \in I_N$ is primitive, i.e. $\operatorname{gcd}(a, b, N)=1$, then we have
$$r_{\mathscr{D}}(e^{[a, b]}_N)=-\frac{1}{2N^2} \sum_{(c, d) \in [a, b]/{\pm 1}} \left(\widetilde{F}\left(  \frac{\langle c\rangle}{N},   \frac{\langle d \rangle}{N}\right) - \widetilde{F}\left(  \frac{\langle -c \rangle}{N},   \frac{\langle -d\rangle}{N}\right) \right)(\widetilde{\omega}^{c, d}- \widetilde{\omega}^{-c, -d}),$$
where $\widetilde{F}(\alpha, \beta)$ is as in the introduction. 
\end{thm}

\section{Regulator of elliptic curves with complex multiplication}
Let $E$ be an elliptic curve over $\Q$ with complex multiplication. 
 Bloch \cite{Bloch} (cf. \cite{Deninger}) constructs an element in the motivic cohomology of $E$ using the $C$-torsion points of $E$ for a rational integer $C$ divisible by the conductor $\bold{f}$ of the Hecke character, and proves that the special value of the $L$-function is expressed in terms of its regulator. 
Otsubo \cite{Otsubo1} modifies Bloch's result only using the $\bold{f}$-torsion points under the assumption \eqref{assump} below.  
In this paper, we use Otsubo's statement.  
 Let $K$ be an imaginary quadratic field such that $E$ has complex multiplication by the whole integer ring $\mathscr{O}_K$. 
Then the class number of $K$ is one. 
Write $\mu_K=\mathscr{O}_K^*$ for the roots of unity in $K$. 
We view $K$ as a subfield of $\C$. 
By results of Deuring, there exists a Hecke character $\chi$ from the group of ideals of $K$ prime to the conductor of $E$ such that
$$L(h^1(E), s)=L(\chi, s)= L(\overline{\chi}, s).$$
Note that $\overline{\chi}(\bold{a})= \chi(\overline{\bold{a}})$ since $E$ is defined over $\Q$. 
Let $\bold{f}$ be the conductor of $\chi$. 
We choose $\chi$ so that $\chi((\alpha))=\overline{\alpha}$ for $\alpha \in \mathscr{O}_K$ with $\alpha \equiv 1$ (mod $\bold{f}$).
Let 
$$\chi_f \colon (\mathscr{O}_K/\bold{f})^* \to K^*$$
be the finite character associated to $\chi$, i.e. 
$$\chi((\alpha))=\overline{\alpha}\chi_f(\alpha)$$
for any $\alpha \in \mathscr{O}_K$ prime to $\bold{f}$. 
We define $\chi(\bold{a})=0$ (resp. $\chi_f(\alpha)=0$) if $\bold{a}$ (resp. an integer $\alpha$) is not prime to $\bold{f}$. 
Let $\omega_E \in H^0(E(\C), \Omega^1)^{+}$ be the real holomorphic differential form normalized so that 
$$\frac1{2\pi i}\int_{E(\C)} \omega_E \wedge \overline{\omega}_E=-1,$$
where $\overline{\omega}_E:=c_{\infty}\omega_E=F_{\infty}\omega_E$. 
Let $\Gamma \subset \C$ be its period lattice, i.e.
$$\Gamma=\left\{\int_{\gamma} \omega_E \mid \gamma \in H_1(E(\C), \Z)\right\}.$$ 
Then we have an analytic isomorphism 
$$E(\C) \xrightarrow{\sim} \C / \Gamma; \quad x \to \int_O^x\omega_E,$$
where $O$ is the origin of the elliptic curve $E$.  
Choose $\Omega \in \Gamma$ such that $\mathscr{O}_K \Omega = \Gamma$, and $\nu \in \mathscr{O}_K$ such that $\bold{f}=(\nu)$. 
Let $E(\R)^0$ be the connected component of the origin with the orientation such that the real period 
$$\Omega_{\R}= \int_{E(\R)^0} \omega_E$$
is positive, and let $\Omega_{\R}=h\Omega$ with $h \in \mathscr{O}_K$. 
Let $E_{\bold{f}}$ denote the $\bold{f}$-torsion points on $E$. By the identification 
$$E_{\bold{f}} \simeq \mathscr{O}_K/\bold{f}; \quad x \mapsto x \overline{\nu}/\Omega,$$
 we view $\chi_f$ also as a character of $E_{\bold{f}}$. 

Let $E_{\rm tor}$ be the torsion subgroup of $E$ and
 $\operatorname{Div}^0(E_{\rm tor})$ be the group of divisors with degree $0$ on $E_{\rm tor}$ defined over $\Q$. 
For $\alpha \in \operatorname{Div}^0(E_{\rm tor})$, 
there exists a non-zero integer $n$ and a rational function $f \in \Q(E)^*$ such that $\dv(f)=n\alpha$. 
If we put 
$$\fa=f \otimes \frac1{n} \in \Q(E)^* \otimes \Q, $$
then $\fa$ is well-defined modulo $\Q^* \otimes \Q$, independently of the choices of $n$ and $f$. 
For $\alpha, \beta \in \operatorname{Div}^0(E_{\rm tor})$, 
$$e_0(\alpha, \beta): =\{\fa, \fb\} \in K_2^M(\Q(E)) \otimes \Q$$
is well-defined modulo the image of $(\Q(E)^* \otimes \Q^*) \otimes \Q$. 
Then there exist a number field $k$, $h_i \in k(E)^*$ and $c_i \in k^*$ such that 
  $$e(\alpha, \beta):=e_0(\alpha, \beta)+N_{k/\Q}\left( \sum\{h_i, c_i\}\right) \in H^2_{\mathscr{M}}(E, \Q(2))_{\Z} $$
(\cite[Lect. 10]{Bloch}, cf. \cite[Section 5]{Deninger}).

\begin{thm}[{\cite[Theorem 4.1]{Otsubo1}}, cf. \cite{Bloch}, { \cite[(3.2)]{Deninger}}] \label{Bloch}
Assume that we can choose $\Omega$ and $\nu$ so that 
\begin{align} 
\frac{\Omega}{\overline{\nu}} \in \R. \label{assump}
\end{align}
Put $\alpha=\displaystyle \sum_{x \in E_{\bold{f}}}([x]-[O])$, $\beta=\displaystyle \sum_{x \in E_{\bold{f}}/\mu_K}([\chi_f(\overline{x})x]-[O]) \in \operatorname{Div}^0(E_{\rm tor})$, and
 the Bloch element $e_{E}:=e(\alpha, \beta)$.  
Then, we have 
$$r_{\mathscr{D}}(e_E)=\pm \frac{N(\bold{f})^{\frac12}}{2|h|} L^*(\chi, 0) \cdot \Omega_{\R} (\omega_E - \overline{\omega}_E). $$
\end{thm}

\section{Comparison}
We compare Bloch's element with Ross's element by using the Bloch map as defined below. 
We have $3N'$ points on $X_{N'}$ 
$$O_n^{\prime}=(0:\zeta^n:1), \quad P_n^{\prime}=(\zeta^n:0:1), \quad Q_n^{\prime}=(1:\xi\zeta^n:0) \quad (n \in \Z/N'\Z),$$
where we put $\zeta=\exp(2 \pi i/N')$ and $\xi=\exp(\pi i/N')$.   
If we choose $O_0^{\prime}$ as the base point, then these points are torsion in the Jacobian \cite{GR}.  

The cup-product pairing is computed as follows (cf. \cite[Proposition 4.2]{Otsubo2}):
\begin{align} \label{cup}
\langle \tw^{a, b}, \tw^{c, d} \rangle = \left\{
\begin{array}{ll}
\dfrac{{N'}^2}{2\pi i} \cdot \dfrac{(1-\zeta^a)(1-\zeta^b)}{1-\zeta^{a+b}} & (c, d)=(-a, -b),\\
0 &   otherwise.
\end{array}
\right.
\end{align}

\subsection{The case $N=36$}
Let $N'=6$, $K=\Q(\mu_6)=\Q(\mu_3)$ and $\mathscr{O}_K$ be its integer ring. 
Let $E_{36}$ be an elliptic curve of conductor 36 over $\Q$ defined by
$$v_0^2w_0 = u_0^3+w_0^3, $$
which is naturally a quotient of $X_6$ by the morphism 
$$p:X_6 \to E_{36}; \quad (u_0 : v_0 : w_0)=( -y_0^2z_0 : x_0^3 : z_0^3). $$
It has complex multiplication by $\mathscr{O}_K$ induced by the multiplication of $\mu_6$ on $y_0$, i.e. 
$\zeta^2 \cdot (u, v)=(\zeta^2 u, v)$.  
The affine equation is written as
$$v^2=u^3+1 \quad (u=u_0/w_0, v=v_0/w_0). $$ 
We easily show that  
$$p^*\left( \frac{du}{2v} \right)=\omega^{3, 2},$$
hence it follows that $p^*$ induces an isomorphism 
$$h^1(E_{36}) \simeq X_6^{[3, 2]}$$
of motives over $\Q$ with $\Q$-coefficients. 
Since the degree of the morphism $p$ is $6$, we have 
$\langle p^*\omega_{E_{36}}, p^*\overline{\omega}_{E_{36}} \rangle = 6\langle \omega_{E_{36}}, \overline{\omega}_{E_{36}}\rangle=-6$. Hence  by \eqref{cup}, we have
$$p^*\omega_{E_{36}} = \sqrt{\frac{\pi}{6\sqrt{3}}} \cdot \widetilde{\omega}^{3, 2}.$$ 

\begin{lem} \label{hom1}
We have $H_1({E_{36}}(\C), \Z) \cong \mathscr{O}_K \cdot p_*\kappa$.  
\end{lem}

\begin{proof}
We prove that $\{p_*\K, \zeta^2 \cdot  p_* \K\}$ is a symplectic basis of $H_1({E_{36}}(\C), \Z)$, where $\zeta^2  \cdot p_* \kappa$ is a cycle induced by the action $\zeta^2 \cdot (u, v) = (\zeta^2u, v)$. 
For $\gamma$, $\gamma' \in H_1(X(\C), \Z)$, let $\gamma \sharp \gamma' \in \Z$ denote the intersection number.  
Note that $X_6$ is generically Galois over $E_{36}$ and $\operatorname{Gal}(X_6/E_{36}) =\langle g^{2, 3}\rangle$. 
Therefore we have 
\begin{align} \label{int1}
 (p_* \K) \s (\zeta^2 \cdot p_* \K )&= \sum_{(r, s) \in [2, 3]}(g^{r, s} \K) \s (g^{0, 1} \K) 
\end{align}
by the projection formula. 
The element $\ok$ in \cite{Otsubo4} agrees with our $\K$, hence 
it follows from Corollary 3.4 of loc. cit. that the right hand side of \eqref{int1} is equal to $1$, 
which means that 
$\{p_*\K$, $\zeta^2 \cdot  p_*\K\}$ is a symplectic basis of $H_1({E_{36}}(\C), \Z)$. 
\end{proof}

Since $\int_{p_*\kappa} \omega_{E_{36}}=\int_{\kappa} p^*\omega_{E_{36}}$
and 
$\int_{\kappa}\tw^{3, 2}=(1-\zeta^3)(1-\zeta^2)=2(1-\zeta^2)$, 
we obtain $\Gamma = \mathscr{O}_K \Omega$ with

$$\Omega=  \sqrt{\frac{\pi}{6\sqrt{3}}} \cdot 2(1-\zeta^2), \quad \Omega_{\R} =(1-\overline{\zeta^2}) \Omega= \sqrt{\frac{6\pi}{\sqrt{3}}}$$
by Lemma \ref{hom1}. 
Note that  we have the conductor $\bold{f}=(2(1-\zeta^2))=(2(1-\overline{\zeta^2}))$(cf. \cite[Section 4.1]{Otsubo3}), so if we let $\nu=2(1-\overline{\zeta^2})$, the assumption 
(\ref{assump}) is satisfied. 
Let $e_6$ be the Ross element.  
\begin{lem}
We have 
$$p_*(e_6)= \{1-v, 1+u\}.$$ 
\end{lem}
\begin{proof}
Let $C$ be a projective curve over $\Q$ defined by an affine equation 
$$v^2+y^6=1. $$
Then there are morphisms
\begin{align*}
&q \colon X_{6} \to C ;  \quad (x, y) \mapsto  (v, y )=(x^3,  y), \\
&r \colon C \to E_{36}; \quad (v,  y) \mapsto (u, v)=(-y^2, v). 
\end{align*}  
We have 
\begin{align*}
&p_*(\{1-x, 1-y\})=r_*q_*(\{1-x, 1-y\})  = r_*(\{q_*(1-x), 1-y\}) \\
&=r_*(\{1-x^3, 1-y\}) 
=\{1-x^3, r_*(1-y)\} 
=\{1-v, 1+u\}. 
\end{align*}
\end{proof}
The following proposition is the key to the main theorem. 
\begin{prop} \label{keyprop1}
Let $e_{E_{36}}$ be the Bloch element and $e_6$ be the Ross element. Then we have $e_{E_{36}}=2p_*(e_6)$ in $H_{\mathscr{M}}^2({E_{36}}, \Q(2))_{\Z}$. 
\end{prop}
To prove the proposition, we use the Bloch map defined as follows. 
Let $E$ be an elliptic curve over $\Q$ and 
$\Z[E(\overline{\Q})]$ be the group algebra of $E(\overline{\Q})$.
We define the Bloch map $\beta$ by
$$\beta \colon \overline{\Q}(E)^* \otimes_{\Z} \overline{\Q}(E)^* \to \Z[E(\overline{\Q})]; \quad  f \otimes g \mapsto \sum_{i, j} m_in_j[p_i-q_j],$$
where $\dv(f)=\sum_im_i[p_i]$ and $\dv(g)=\sum_jn_j[q_j]$ are divisors of $f$ and $g$ respectively.  
Let $R_3(E)$ be the subgroup of $\Z[E(\overline{\Q})]$ generated by the divisors $\beta (f \otimes (1-f))$ with
$f \in \overline{\Q}(E)$, $f \neq 0, 1$ as well as the divisors $[p]+[-p]$ with $p \in E(\overline{\Q})$. 
Brunault \cite[Definition B.3]{AC} defines a modification of the Bloch group, which is originally defined by Goncharov-Levin \cite[Definition 3.1]{GL},  
by 
$$B_3(E)=\left(\Z[E(\overline{\Q})]/R_3(E)\right)^{\operatorname{Gal}(\overline{\Q}/\Q)}.$$
In the group $B_3(E) \otimes \Q$, we have the relation $[p]+[-p]=0$ for any point $p$, hence $[p]=0$ if $p$ is a $2$-torsion point.  
Then $\beta$ induces a map
$$\overline{\beta} \colon K_2(E) \otimes \Q \to B_3(E) \otimes \Q$$ 
and $\overline{\beta}$ is injective (see \cite[Theorem B.5]{AC}). 
Hence any equality in $H_{\mathscr{M}}^2(E, \Q(2))$ can be proved by comparing the divisors. 
\begin{proof}[Proof of Proposition \ref{keyprop1}]
Let $O=(-1 : 0 : 1)$, $P=(0 : 1: 1)$ and $Q=(0 : 1: 0)$ be the images under $p$ of $O_0^{\prime}$, $P_0^{\prime}$ and $Q_0^{\prime}$ respectively, and $O$ be the origin of $E_{36}$. 
Note that $(\mathscr{O}_K/\bold{f})^*/\mu_K=\{1\}$, and 
since 
$$\int_O^P\omega_{E_{36}}= \int_{O'_0}^{P'_0} p^*\omega_{E_{36}}= \dfrac{\Omega}{\overline{\nu}} \int_{O'_0}^{P'_0}\tw^{3, 2}=\dfrac{\Omega}{\overline{\nu}}, $$
$P$ corresponds to 1 under $E_{\bold{f}} \simeq \mathscr{O}_K/\bold{f}$. 
Then the divisors of Theorem \ref{Bloch} are
$$
\dv(f_{\alpha})=\sum_{x \in E_{\bold{f}}}[x]-12[O], \quad 
\dv(f_{\beta})=[P]-[O].
$$
Applying the map $\overline{\beta}$ to $e_0(\alpha, \beta)=\{\fa, \fb\}$, we have
\begin{align*} 
\overline{\beta}(e_0(\alpha, \beta))&=\sum_{x \in E_{\bold{f}}}[x-P] -\sum_{x \in E_{\bold{f}}}[x-O]
-12[O-P]+12[O-O] \\
&=\sum_{x \in E_{\bold{f}}}[x-P] -\sum_{x \in E_{\bold{f}}}[x]
-12[-P]+12[O].
\end{align*}
We note that $P$ is an $\bold{f}$-torsion point on $E_{36}$, hence we have 
$$\sum_{x \in E_{\bold{f}}}[x-P] =\sum_{x \in E_{\bold{f}}}[x].$$
 Since $O$ is a $2$-torsion point and 
 $-[P]=[-P]$ in $B_3(E_{36}) \otimes \Q$,  we have 
\begin{align}
\overline{\beta}(e_0(\alpha, \beta))&=12[P]. \label{beta1}
\end{align}
On the other hand, 
the divisors of $1-v$ and $1+u$ are
\begin{align*}
\operatorname{div}(1-v)=3([P]-[Q]), \quad 
\operatorname{div}(1+u)=2([O]-[Q]). 
\end{align*}
Applying the map $\overline{\beta}$ to $p_*(e_6)$, we have
\begin{align*}
\overline{\beta}(p_*(e_6))&=6([P-O] +[Q-Q]-[P-Q]-[Q-O] 
) \\
&=6([P] +[O]-[R]-[Q] ),
\end{align*}
where we put $R=(2 : -3 : 1)$. 
Note that $Q$ is a 2-torsion point, we have
\begin{align*}
\overline{\beta}(p_*(e_6))=6([P] -[R] ).
\end{align*}
We are to show  that $[R]=0$ in $B_3(E_{36})$ (cf. \cite[Section 4.2]{Mellit}). 
Put 
$$\{f, g\}=\left\{\frac12 (1-v), \frac12(1+v)\right\}. $$
Then we have $\overline{\beta}(\{f, g\})=0$ since $f+g=1$. 
The divisors of $f$ and $g$ are
\begin{align*}
\operatorname{div}(f)=3([P]-[Q]), \quad 
\operatorname{div}(g)=3([-P]-[Q]), 
\end{align*}
hence 
\begin{align*}
\overline{\beta}(\{f, g\})&=9([P+P] +[P-Q]-[Q+P]-[Q-Q] ) \\
&=9([-R] -2[R] )=-27[R]=0,
\end{align*}
which concludes that $[R]=0$. 
Therefore, we obtain
\begin{align}
\overline{\beta}(p_*(e_6))=6[P]. \label{beta2}
\end{align}
For a number field $k$, a symbol of the form $N_{k/\Q}\left( \sum\{h_i, c_i\}\right)$ 
($h_i \in k(E_{36})^*$ and $c_i \in k^*$)
is killed by $\overline{\beta}$. 
For the image of $N_{k/\Q}(\{h_i, c_i\})$ in $K_2^M(\overline{\Q}(E_{36}))$ is $\sum_{\sigma : k \hookrightarrow \overline{\Q}}\{\sigma(h_i), \sigma(c_i)\}$, which is killed by $\beta$ since $\sigma(c_i)$ is constant.  
Therefore we have 
$$\overline{\beta}(e_{E_{36}})
=\overline{\beta}(e_0(\alpha, \beta))
=2\overline{\beta}(p_*(e_6)) $$ 
by comparing (\ref{beta1}) and $(\ref{beta2})$. 
Since $\overline{\beta}$ is injective, it concludes that 
$$e_{E_{36}}=2p_*(e_6) \in H_{\mathscr{M}}^2({E_{36}}, \Q(2))_{\Z},$$
which finishes the proof. 
\end{proof}

\begin{rmk}
We can prove that
$$\reg(e_{E_{36}})=2\reg(p_*(e_6))$$
without Proposition \ref{keyprop1}. 
We can take 
$$\fa=\frac{(1+u^3)^3(1-v^2)^2(8-u^3)^6}{(1+u)^{36}} \otimes \frac16, \quad \fb=\frac{(1-v)^2}{(1+u)^3} \otimes \frac16. $$
The regulator map $\reg$ extends to $K_2^M(\C(E_{36}))$  (cf. \cite[Section 1]{Ross}), hence  
we have
\begin{align*}
&36\reg(e_{E_{36}})=\reg(6\{1+u^3, 1-v\}+4\{1-v^2, 1-v\}+12\{8-u^3, 1-v\}\\
& -9\{1+u^3, 1+u\}-6\{1-v^2, 1+u\}-18\{8-u^3, 1+u\}-72\{1+u, 1-v\}) 
\end{align*}
 since $\reg \left(N_{k/\Q}(\{h_i, c_i\})\right)= \sum_{\sigma:k \hookrightarrow \C}\reg \left(\{\sigma(h_i), \sigma(c_i)\}\right)=0$.  
Note that $\{1+u^3, 1-v\}=\{v^2, 1-v\}=0$, $\{1-v^2, 1+u\}=\{-u^3, 1+u\}=0$. 
Also 
$\{8-u^3, 1-v\}=\{9-v^2, 1-v\}$, 
$\{1-v^2, 1-v\}$, $\{1+u^3, 1+u\}$ and $\{8-u^3, 1+u\}$
is killed by $\reg$ 
since these elements come from a quotient rational curve (\cite[Lemma 1]{Ross}). 
Hence we have
$$\reg(e_{E_{36}})=2\reg(\{1-v, 1+u\})=2\reg(p_*(e_6)). $$
\end{rmk}

\begin{thm} \label{main1}
Let $E_{36}$ be an elliptic curve of conductor 36 over $\Q$. Then we have
$$L^*(E_{36}, 0)=\frac{1}{2\sqrt{3}\pi}\left(\widetilde{F}\left( \frac12, \frac13 \right) - \widetilde{F}\left( \frac12, \frac23 \right) \right).$$
\end{thm}
\begin{proof}
It is known that the isogeny class over $\Q$ of $E_{36}$ is unique \cite{LMDBF}, hence it suffices to show the case $E_{36}:v^2=u^3+1$.  
By Theorem \ref{Bloch}, we have
$$r_{\mathscr{D}}(e_{E_{36}})=\pm L^*(E_{36}, 0) \cdot  \sqrt{\frac{6\pi}{\sqrt{3}}}(\omega_{E_{36}}-\overline{\omega}_{E_{36}}). $$
By applying $p^*$ on the both sides and using Proposition \ref{keyprop1}, we have
\begin{align*}
r_{\mathscr{D}}(p^*p_*(e_6))&=\pm\frac12 \cdot \sqrt{\frac{6\pi}{\sqrt{3}} }L^*(E_{36}, 0)(p^*\omega_{E_{36}} -p^*\overline{\omega}_{E_{36}}) \\
&=\pm\frac12 \cdot \sqrt{\frac{6\pi}{\sqrt{3}} }\cdot  \sqrt{\frac{\pi}{6\sqrt{3}}}L^*(E_{36}, 0)(\tw^{3, 2}-\tw^{3, 4}). 
\end{align*}
We note that 
$$p^*p_*(e_6)=6p^{[3,  2]}e_6=6e_6^{[3, 2]},$$ 
hence by Theorem \ref{Otsubo}, 
we obtain the equality except for the sign.
Since $L(E_{36}, 2)$ is positive and the root number (the sign of the functional equation) is $1$, $L^*(E_{36}, 0)$ is also positive. 
It is known that $\widetilde{F}(\alpha, \beta)$ is monotonously decreasing with respect to each parameter \cite[Proposition 4.25]{Otsubo}, hence the right hand side is also positive. 
\end{proof}

\begin{rmk}
Note that each of $E_{108}$, $E_{144}$ and $E_{432}$ is also a quotient of the Fermat curve of degree 6. 
However, we cannot use the same method for $E_{108}$ and $E_{144}$ since 
these do not satisfy the assumption (\ref{assump}). 
Although $E_{432}$ satisfies the assumption (\ref{assump}), it seems to be difficult to compare Bloch's element with Ross's element in $H^2_{\mathscr{M}}(E_{432}, \Q(2))_{\Z}$ since these are so complicated. 
\end{rmk}

\subsection{The case $N=64$}
Let $K=\Q(\mu_4)$ and $\mathscr{O}_K$ be its integer ring. 
Let $E_{64}$ be an elliptic curve of conductor 64 over $\Q$ defined by
$$v_0^2w_0=u_0^3-4u_0w_0^2, $$
which is naturally a quotient of $X_4$ by the morphism 
$$p:X_4 \to E_{64}; \quad (u_0 : v_0 : w_0)=(2x_0(y_0^2+z_0^2): 4y_0(y_0^2+z_0^2): x_0^3). $$
It has complex multiplication by $\mathscr{O}_K$ induced by the multiplication of $\mu_4$ on $x_0$, i.e. $\zeta \cdot (u, v)=(- u, \zeta v)$. 
The affine equation is written as
$$v^2=u^3-4u \quad (u=u_0/w_0, \ v=v_0/w_0). $$
We easily show that  
$$p^*\left( \frac{du}{2v} \right)=-\frac12\omega^{1, 1},$$
hence it follows that $p^*$ induces an isomorphism 
$$h^1(E_{64}) \simeq X_4^{[1, 1]}$$
of motives over $\Q$ with $\Q$-coefficients. 
Since the degree of the morphism $p$ is 2, we have $\langle p^*\omega_{E_{64}}, p^*\overline{\omega}_{E_{64}} \rangle = 2\langle \omega_{E_{64}}, \overline{\omega}_{E_{64}}\rangle=-2$. Hence by \eqref{cup}, we have
$$p^*\omega_{E_{64}} = \frac{\sqrt{\pi}}{2} \widetilde{\omega}^{1, 1}.$$ 

\begin{lem} \label{hom2}
We have $H_1({E_{64}}(\C), \Z)=\mathscr{O}_K \cdot p_*\kappa$. 
\end{lem}
\begin{proof}
We prove that $\{p_*\K, \zeta \cdot  p_* \K\}$ is a symplectic basis of $H_1({E_{64}}(\C), \Z)$, where $\zeta  \cdot p_* \kappa$ is a cycle induced by the action $\zeta \cdot (u, v) = (-u, \zeta v)$. 
Note that $X_4$ is generically Galois over $E_{64}$ and $\operatorname{Gal}(X_4/E_{64}) = \langle g^{2, 2}\rangle$. 
Therefore we have 
\begin{align} \label{int2}
 (p_* \K) \s (\zeta \cdot p_* \K )&= \sum_{(r, s) \in [2, 2]}(g^{r, s} \K) \s (g^{1, 0} \K) 
\end{align}
by the projection formula. 
It follows from \cite[Corollary 3.4]{Otsubo4} that the right hand side of \eqref{int2} is equal to $1$, 
which means that 
$\{p_*\K$, $\zeta \cdot  p_*\K\}$ is a symplectic basis of $H_1({E_{64}}(\C), \Z)$. 
\end{proof}

Since $\int_{p_*\kappa} \omega_{E_{64}}=\int_{\kappa} p^*\omega_{E_{64}}$
and 
$\int_{\kappa}\tw^{1, 1}=(1-i)(1-i)=-2i$, 
we obtain $\Gamma = \mathscr{O}_K \Omega$ with
$$\Omega=  \Omega_{\R} = \sqrt{\pi}$$
by Lemma \ref{hom2}. 
Note that  we have the conductor $\bold{f}=(4)$, so if we let $\nu=4$, the assumption (\ref{assump}) is satisfied. 
Let $e_4$ be the Ross element.  

\begin{lem} \label{element}
We have 
$$p_*(e_4)= \left\{ \frac{v-2u}v, \frac{32u^2}{(u-2)^2v^2}\right\}+\left\{\frac{(u-2)^2}{u^2+4}, -\frac{v}{2u}\right\}. $$
\end{lem}
To prove the lemma, we use the Rosset-Tate algorithm \cite{RT} as follows. 
Let $E \subset F$ be a finite field extension and 
$$\operatorname{Tr}_{F/E} \colon K_2^M(F) \to K_2^M(E)$$
be the trace map. 
For a polynomial 
$$f(T)=a_nT^n+a_{n-1}T^{n-1}+ \cdots + a_mT^m$$
where $n \geq m$ and $a_ma_n \neq 0$, 
put 
$$f^*(T)=(a_mT^m)^{-1}f(T), \quad c(f)=(-1)^na_n. $$

\begin{prop}[{\cite[Section 3]{RT}}] \label{RT}
Let $E \subset F$ be a finite extension of fields. For $x, y \in F^*$,
let $g \in E[T]$ be
the monic irreducible polynomial with root $x$ and 
$f \in E[T]$ be the polynomial of smallest degree such that $N_{F/E(x)} y = f(x)$.  
Let $g_0, g_1, \ldots, g_m \neq 0$, $g_{m+1}=0$ be the sequence of polynomials of strictly decreasing degree defined by
$$g_0=g, \quad g_1=f$$
and for $i \geq 1$, 
$$g_{i+1} \equiv g_{i-1}^* \pmod{g_i}$$
provided $g_i \neq 0$. 
Then we have
$$\operatorname{Tr}_{F/E}(\{x, y\})=-\sum_{i=1}^m\{c(g_{i-1}^*), c(g_i)\}.$$
\end{prop}

\begin{proof}[Proof of Lemma \ref{element}]
We regard the function field $\Q(E_{64})=\Q(u, v)$ as a subfield of $\Q(X_{4})=\Q(x, y)$ by the inclusion map induced by $p$.
Put $E=\Q(u, v)$ and $F=\Q(x, y)$. 
Then 
the minimal polynomial of $1-x$ is given by
$$g_0(T)=(T-1)^2-\frac{4u}{u^2+4}, \quad g_0^*(T)=\frac{u^2+4}{(u-2)^2}g_0(T).$$
Note that $E(1-x)=F$, hence we require $g_1\in E[T]$ to be the polynomial such that $g_1(1-x)=1-y$. 
One sees easily that
$$g_1(T)=\frac{v}{2u}T+1-\frac{v}{2u}, \quad g_1^*(T)=\frac{2u}{2u-v}g_1(T). $$
By the definition of $g_i$, we have
$$g_2(T)=\frac{32u^2}{v^2(u-2)^2}, \quad g_i(T)=0\ (i \geq 3).$$
By Proposition \ref{RT}, we obtain
\begin{align*}
p_*(e_4)&=-\{c(g_0^*), c(g_1)\}-\{c(g_1^*), c(g_2)\} \\
&=-\left\{\frac{u^2+4}{(u-2)^2}, -\frac{v}{2u}\right\}
-
\left\{\frac{v}{v-2u}, \frac{32u^2}{(u-2)^2v^2}\right\} \\
&= \left\{ \frac{v-2u}v, \frac{32u^2}{(u-2)^2v^2}\right\}+\left\{\frac{(u-2)^2}{u^2+4}, -\frac{v}{2u}\right\}, 
\end{align*}
which finishes the proof.
\end{proof}

\begin{prop} \label{keyprop2}
Let $e_{E_{64}}$ be the Bloch element and $e_4$ be the Ross element. Then we have 
$e_{E_{64}}=2p_*(e_4)$ in $H_{\mathscr{M}}(E_{64}, \Q(2))_{\Z}$. 
\end{prop}

\begin{proof}
Let $O$ (resp. $P_i$, $Q_i$) be the image under $p$ of $O_0^{\prime}$ (resp. $P_i^{\prime}$, $Q_i^{\prime}$), and  
$O$ be the origin of $E_{64}$. 
We put 
$$R=(0 : 0 :1), \quad S=(2+2\sqrt{2} :4+4\sqrt{2} : 1), \quad T=(2 -2\sqrt{2} : 4-4\sqrt{2} : 1)$$
 on $E_{64}$.
Note that $(\mathscr{O}_K/\bold{f})^*/\mu_4=\{1, 1-2i\}$, and 
since 
$$\int_O^{P_0}\omega_{E_{64}}= \int_{O'_0}^{P'_0} p^*\omega_{E_{64}}= \dfrac{2\Omega}{\overline{\nu}} \int_{O'_0}^{P'_0}\tw^{1, 1}=\dfrac{2\Omega}{\overline{\nu}}, $$
$P_0$ corresponds to $2$ under $E_{\bold{f}} \simeq \mathscr{O}_K/\bold{f}$. 
One can show that $2S=2T=P_0$, hence $S$, $T$ correspond to one of 
$$\pm1, \quad \pm1 \pm 2i \quad $$
under $E_{\bold{f}} \simeq \mathscr{O}_K/\bold{f}$. 
By using Mathematica, we verify that 
$S$, $T$ correspond to $1$, $1-2i$ respectively. 
By \cite[Section 4.1]{Otsubo3}, we have $\chi_{f}(1-2i)=1$, 
hence
the divisors of Theorem \ref{Bloch} are
$$\dv(\fa)=\sum_{x \in E_{\bold{f}}}[x]-16[O], \quad \dv(\fb)=[S]+[T]-2[O].$$
Applying the map $\overline{\beta}$ to $e_0(\alpha, \beta)=\{\fa, \fb\}$, we have
\begin{align*}
\overline{\beta}(e_0(\alpha, \beta)) 
&=\sum_{x \in E_{\bold{f}}}[x-S]+\sum_{x \in E_{\bold{f}}}[x-T]-2\sum_{x \in E_{\bold{f}}}[x] 
-16[-S]-16[-T]+32[O].
\end{align*}
We note that $S$, $T$ are $\bold{f}$-torsion points, hence we have
$$\sum_{x \in E_{\bold{f}}}[x-S]=\sum_{x \in E_{\bold{f}}}[x-T]=\sum_{x \in E_{\bold{f}}}[x]. $$
Since $-[S] = [-S]$, $-[T]=[-T]$ and $[O]=0$ in $B_3(E_{64}) \otimes \Q$, we have
\begin{align} \label{beta3}
\overline{\beta}(e_0(\alpha, \beta))=16([S]+[T]). 
\end{align} 
On the other hand, we put 
$$f_1=\frac{v-2u}{v}, \quad g_1=\frac{32u^2}{(u-2)^2v^2}.$$
Then the divisors of $f_1$ and $g_1$ are
\begin{align*}
\operatorname{div}(f_1)&=[S]+[T]-[P_0]-[P_1], \quad  \operatorname{div}(g_1)=2[R]+6[O]-6(P_0)-2(P_1).
\end{align*}
Applying the map $\overline{\beta}$ to $\{f_1, g_1\}$, we have
\begin{align*}
\overline{\beta}(\{f_1, g_1\}) &=2[S-R]+6[S]+2[T-R]+6[T]+6[O]+2[P_0-P_1]+2[O] \\
&+6[P_1-P_0] -6[S-P_0]-2[S-P_1]-6[T-P_0]-2[T-P_1] \\
&-2[P_0-R] -6[P_0]-2[P_1-R]-6[P_1].
\end{align*}
We compute that 
$$S-R=T-P_0=P_1-S=-T, \quad  T-R=S-P_0=P_1-T=-S, $$
$$P_0-P_1=R, \quad  P_0-R=P_1, \quad  P_1-R=P_0, $$
hence we have
$$\overline{\beta}(\{f_1, g_1\})=8([S]+[T]+[O]+[P_0]+[P_1])-4[R].$$
We note that the points $O$, $P_0$, $P_1$ and $R$ are $2$-torsion points, hence we have $[O]=[P_0]=[P_1]=[R]=0$ in $B_3(E_{64}) \otimes \Q$. 
Therefore we obtain
\begin{align} \label{beta4}
\overline{\beta}(\{f_1, g_1\})= 8([S]+[T]). 
\end{align}
We put 
$$f_2= \frac{(u-2)^2}{u^2+4}, \quad g_2=-\frac{v}{2u}.$$
The divisors of $f_2$ and $g_2$ are
\begin{align*}
&\operatorname{div}(f_2)=2[P_0]+2[O]-[Q_0]-[-Q_0]-[Q_3]-[-Q_3], \\
&\operatorname{div}(g_2)=[P_0]+[P_1]-[R]-[O],
\end{align*}
hence we can prove similarly that $\overline{\beta}(\{f_2, g_2\})=0$  as above.
Hence we have as before 
$$\overline{\beta}(e_{E_{64}})=\overline{\beta}(e_0(\alpha, \beta))
=2\overline{\beta}(p_*(e_4)) $$ 
by comparing $(\ref{beta3})$ and $(\ref{beta4})$. 
Since $\overline{\beta}$ is injective, it concludes that 
$$e_{E_{64}}=2p_*(e_4) \in H^2_{\mathscr{M}}(E_{64}, \Q(2))_{\Z}, $$
which finishes the proof. 
\end{proof}

\begin{thm} \label{main2}
Let $E_{64}$ be an elliptic curve of conductor 64 over $\Q$. Then we have
$$L^*(E_{64}, 0)=\frac{1}{8\pi}\left(\widetilde{F}\left( \frac14, \frac14 \right) - \widetilde{F}\left( \frac34, \frac34 \right) \right).$$
\end{thm}

\begin{proof}
It is known that the isogeny class over $\Q$ of $E_{64}$ is unique \cite{LMDBF}, hence it suffices to show the case $E_{64}:v^2=u^3-4u$.  
We note that $p^*p_*(e_4)=4p^{[1,  1]}e_4=4e_4^{[1, 1]}$. Therefore, by Proposition \ref{keyprop2} and the fact that the root number is 1, 
we can prove the theorem similarly as in the proof of Theorem \ref{main1}. 
\end{proof}

\section*{Acknowledgment}
The author would like to thank sincerely  Noriyuki Otsubo for valuable discussions and many helpful comments on a draft version of this paper, and 
for his warm and constant encouragement. 
The author would also express his sincere gratitude to the anonymous referee for many helpful comments. 
This paper is a part of the outcome of research performed under Waseda University Grant for Special Research Projects (Project number: 2023C-274) and Kakenhi Applicants (Project number: 2023R-044).

\end{document}